\documentclass[12pt]{amsart}
\usepackage[margin=3cm]{geometry}                % See geometry.pdf to learn the layout options. There are lots.
\usepackage{graphicx}
\usepackage{amssymb}
\usepackage{bbm}
\usepackage{epstopdf}
\usepackage{proof}

\newcommand{\typeof}[1]{|#1|}
\newcommand{\eqof}[1]{=_{#1}}

\newcommand{\rsymb}[0]{{\mathcal R}}
\newcommand{\rprod}[0]{\times}
\newcommand{\rpw}[0]{{\bf P}}
\newcommand{\rin}[0]{\;\epsilon\;}
\newcommand{\rlevel}[1]{|#1|}
\newcommand{\elevel}[1]{||#1||}
\newcommand{\rint}[1]{[#1]^R}

\newcommand{\inmax}[0]{\lor}
\newcommand{\vsymb}[2]{{\sf v}_{#1}^A}

\newcommand{\vf}{\varphi}
\newcommand{\unittype}[0]{{\bf 1}}
\newcommand{\theelement}[0]{\star}
\newcommand{\nattype}[0]{{\bf N}}
\newcommand{\suc}[1]{{\sf S}(#1)}

\newcommand{\pair}[2]{\langle #1, #2 \rangle}
\newcommand{\fst}[1]{ \pi_1(#1)}
\newcommand{\snd}[1]{ \pi_2(#1)}

\newcommand{\terms}[1]{{\rm Term}(#1)}
\newcommand{\formulas}[1]{{\rm Form}(#1)}

\newcommand{\ipret}[1]{{#1}^{*}}

\newcommand{\cutout}[1]{}

\newtheorem{thm}{Theorem}[section]

\newtheorem{lemma}[thm]{Lemma}

\newtheorem{examples}[thm]{Examples}
\newtheorem{remark}[thm]{Remark}

\title{A Constructive Examination of a Russell-style \\
 Ramified Type Theory }
\author{Erik Palmgren }
\date{April 21, 2017}     
\thanks{The author was supported by a grant from the Swedish Research Council (VR). Author's address: Department of Mathematics, Stockholm University, 106 91 Stockholm. Email: palmgren[at]math.su.se. }
\begin{document}

\begin{abstract}
In this paper we examine the natural interpretation of a ramified
type hierarchy into Martin-L{\"o}f type theory with an infinite sequence
of universes. It is shown that under this predicative interpretation some useful
special cases of Russell's reducibility axiom are valid, namely functional reducibility. This is sufficient
to make the type hierarchy usable for development of constructive
mathematics. We present a ramified type theory suitable for this 
purpose. One may regard the results of this paper as an alternative solution to the problems
of Russell's theory, which avoids impredicativity, but instead imposes constructive logic.
The intuitionistic ramified type theory introduced here, also suggests that there is a natural 
associated notion of predicative elementary topos.
\end{abstract}

\maketitle
% Perhaps: include this dedication
%
%\rightline{\em Dedicated to Per Martin-L\"of  on }
%\rightline{\em the occasion of his 75th birthday.}
%\medskip
%\subsection{}

Russell introduced with his ramified type theory a distinction between
different levels of propositions in order to solve logical paradoxes,
notably the Liar Paradox and the paradox he discovered in Frege's system
(Russell 1908). To be able to carry out certain mathematical
constructions, e.g.\ the real number system, he was then compelled to introduce the reducibility
axiom. This had however the effect of collapsing the ramification, from an
extensional point of view, and thus making the system impredicative. The original Russell theory
is not quite up to modern standards of presentation of a formal system: a treatment of substitution
is lacking.
In the article by Kamareddine, Laan and
Nederpelt (2002) however a modern reconstruction of Russell's
type theory using lambda-calculus notation is presented. We refer to their article for further background and history. 

In this paper we shall present an intuitionistic version of
ramified type theory IRTT.  By employing a
restricted form of reducibility it can be  shown to be
predicatively acceptable.  This axiom,
called the Functional Reducibility Axiom (FR),  reduce type levels only
of total functional relations. The axiom is enough to
handle the problem of proliferation of levels of real numbers encountered 
in Russell's original theory (Kamareddine {\em et al.} 2002, pp.\ 231 -- 232). It is essential that theory is based on intuitionistic logic,
as (FR) imply the full reducibility principle using classical logic.   The system IRTT is
demonstrated to be predicative by interpreting it in a subsystem of Martin-L{\"o}f type theory (Martin-L{\"o}f 1984),
a system itself predicative in the proof-theoretic 
sense of Feferman and Sch\"utte (Feferman 1982). 
One may regard the results of this paper as an alternative solution to the problems
of Russell's theory which avoids impredicativity, but instead imposes constructive logic.

The interpretation is carried out in Sections 2 and 3.  In Section 4, 
we see how universal set constructions useful for e.g.\ formalizing real numbers can be 
carried out in IRTT.  The intuitionistic ramified type theory introduced here, also suggests that there is a natural 
associated notion of predicative elementary topos,  but this will have to be developed elsewhere. Adding the principle of excluded middle to IRTT makes  the full reducibility axiom a theorem (Section 5).

An extended abstract of an early version of this paper has been published as \cite{P08}.

\section{Ramified Type Theory} 

Our version of  the ramified type hierarchy is built from basic
types $\unittype$ (the unit type) and $\nattype$ (the type of natural
number) using the product type construction $\times$, and 
for each $n=0,1,2,\ldots$ the
restricted power set operation ${\bf P}_n(S)$. The latter 
operation assigns to each type $S$
the type of level $n$ propositional functions on $S$, 
or level $n$ subsets of $S$.  (See Remark \ref{russelltypes} below for a comparison
with Russell's ramified types.)
Intuitively these power sets form an increasing sequence:
$$\rpw_0(S) \subseteq \rpw_1(S) \subseteq \rpw_2(S) \subseteq \cdots$$
To avoid impredicativity  when 
forming a subset
$$\{x: A \;|\; \vf(x) \} : \rpw_k(A)$$ 
it is required that $\vf(x)$ does
not contain quantifiers over $\rpw_n(S)$ where
$n \ge k$. 
A version of the {\em full reducibility axiom} says that this 
hierarchy collapses from an extensional point of view: for each level $r$,
\begin{equation}\label{fr}
(\forall X: \rpw_r(S))(\exists Y:\rpw_0(S))(\forall z: S)(z \rin X
\Leftrightarrow z \rin Y).
\end{equation}
This has the effect of reintroducing 
impredicativity, as was observed by Ramsey (Ramsey 1926); see also Myhill (1979). However, 
a special case of the reducibility axiom is predicatively acceptable
if we work against the background of intuitionistic logic.
  This is shown by modelling it in Martin-L{\"o}f type theory.

We now turn to the formal presentation of our theory.
The set of {\em ramified type symbols } $\rsymb$ is inductively defined by
\begin{itemize}
\item
 $\unittype, \nattype  \in \rsymb$,
\item if $A,B \in \rsymb$, then $A \rprod B \in \rsymb$,
\item if $A \in \rsymb$ and $n \in {\mathbb N}$, then $\rpw_n(A) \in \rsymb$.
\end{itemize}
The {\em level} of a type symbol $A$, $\rlevel{A}$,  is defined recursively
\begin{eqnarray*}
\rlevel{\unittype} &= & \rlevel{\nattype} = 0 \\
\rlevel{A \times B} &= & \max(\rlevel{A},\rlevel{B})\\
\rlevel{\rpw_n(A)} &= &\max(n+1, \rlevel{A}).
\end{eqnarray*}
For instance, the level of the type expression 
${\bf P}_3(\nattype) \times {\bf P}_1(\nattype \times {\bf P}_1(\unittype))$
is 4. Let $\rsymb_n = \{A \in \rsymb \: : \: \rlevel{A} \le n \}$. 
We also define the {\em equality level}Ê  $\elevel{A}$ of a type $A$ recursively
by 
\begin{eqnarray*}
\elevel{\unittype} &= & \elevel{\nattype} = 0 \\
\elevel{A \times B} &= & \max(\elevel{A},\elevel{B})\\
\elevel{\rpw_n(A)} &= &\max(n, \rlevel{A}).
\end{eqnarray*}
(The significance of this measure is seen in Lemma \ref{elevel_lm}.)
Below we often write $n_1 \inmax \cdots \inmax n_k$ for $\max(n_1,\ldots, n_k)$.

Our system IRTT of {\em intuitionistic ramified type theory} will be based on many-sorted
intuitionistic logic. The sorts will be the types in $\rsymb$.  We define simultaneously the
set of terms $\terms{A}$ of type $A \in \rsymb$ and the set of
formulas of level $k\in {\mathbb N}$, denoted $\formulas{k}$.
\begin{itemize}
\item For each $A \in \rsymb$ there is a countable sequence of variables of sort $A$: 
$\vsymb{0}{A}, \vsymb{1}{A}, \vsymb{2}{A},\ldots$ in $\terms{A}$;
\item $\theelement \in \terms{\unittype}$;
\item $0 \in \terms{\nattype}$;
\item If $a \in \terms{\nattype}$, then $\suc{a} \in \terms{\nattype}$;
\item If $a,b \in \terms{\nattype}$, then $a+b, a \cdot b \in \terms{\nattype}$;
\item If $a \in \terms{A}$, $b \in \terms{B}$, then $\pair{a}{b} \in \terms{A \times B}$;
\item If $c \in \terms{A \times B}$, then $\fst{c} \in \terms{A}$ and $\snd{c} \in \terms{B}$,
\item If $\vf \in \formulas{k}$ and $x$ is a variable in $\terms{A}$, then 
$$\{x : A \; | \; \vf \} \in \terms{ {\rpw_k(A)} } $$
(This is the set-abstraction term and $x$ is considered to be a bound variable in this term.)
\item If $\rlevel{A} \le k$ and $a, b \in \terms{A}$, then $(a =_A b) \in \formulas{k}$;
\item If $a \in \terms{A}$ and $b\in \terms{\rpw_n(A)}$, then $(a \rin b) \in \formulas{k}$ for any $k \ge n$;
\item $\bot \in \formulas{k}$;
\item If $\vf, \psi \in \formulas{k}$, then $(\vf \lor \psi), (\vf \land \psi), (\vf \Rightarrow \psi) \in \formulas{k}$;
\item If  $\vf \in \formulas{k}$ and $x$ is a variable in $\terms{A}$ where $\rlevel{A} \le k$, then
$(\forall x: A)\vf, (\exists x: A)\vf \in \formulas{k}$.
\end{itemize}

It is clear that 
$$\formulas{0} \subseteq \formulas{1} \subseteq \formulas{2} \subseteq \cdots $$

The axioms of ramified type theory are the following.
First there are standard axioms for equality stating that each
$=_A$ is an equivalence relation and that operations  and predicates respect
these equivalence relations.
%\begin{eqnarray*}
%&&x =_A x \\
%&&x =_A y \Longrightarrow y =_A x\\
%&&x =_A y  \land y =_A z \Longrightarrow x =_A z\\
%\end{eqnarray*}
%
%\begin{eqnarray*}
%&& x=_A u \land y=_B v \Longrightarrow \pair{x}{y} =_{A \times B} \pair{u}{v}\\
%&& z =_{A\times B} w \Longrightarrow \fst{z} =_A \fst{w} \land \snd{z} =_B \snd{w}\\
%&& x=_{\nattype} u \Longrightarrow \suc{x} =_{\nattype} \suc{u}\\
%&& x=_{\nattype} u \land y=_{\nattype} v \Longrightarrow x+y =_{\nattype} u+v \land x\cdot y =_{\nattype} u \cdot v\\\
%&& x \rin X \land x=_A y \land X=_{\rpw_k(A)} Y \Longrightarrow y \rin Y \\
%\end{eqnarray*}
Axioms for unit type and product type
\begin{eqnarray*}
z &=_{\unittype} & \theelement \\
\fst{\pair{x}{y}} &=_A& x \\
\snd{\pair{x}{y}} &=_B& y \\
\pair{\fst{z}}{\snd{z}} &=_{A \times B} &z. 
\end{eqnarray*}
The arithmetical axioms are the standard Peano axioms for $0$, $S$, $+$ and $\cdot$,
together with the induction scheme.
%For the successor operation we have
%\begin{eqnarray*}
%&& \lnot 0 = \suc{x} \\
%&& \suc{x} = \suc{y}  \Longrightarrow x = y \\
%\end{eqnarray*}
%The defining axioms for the arithmetical operations are
%\begin{eqnarray*}
%x + 0         &=& x \\
%x + \suc{y} &=& \suc{x+y} \\
%x \cdot 0 &=& 0 \\
%&x \cdot \suc{y} &=& x \cdot y + x \\
%\end{eqnarray*}

%Induction scheme: For any formula $\vf$ and any variable $x \in \terms{\nattype}$ we have 
%$$\vf[0/x] \land (\forall z: \nattype)(\vf[z/x] \Rightarrow \vf[\suc{z}/x]) \Rightarrow (\forall z: \nattype)\vf[z/x].$$

For subsets we have the following axioms
\begin{itemize}
\item[] {\em Axiom of Extensionality:} 
$$(\forall X,Y : \rpw_k(A))((\forall z: A)(z \rin X \Leftrightarrow z \rin Y) \Rightarrow X=_{\rpw_k(A)} Y)$$
\item[] {\em Defining Axiom for Restricted Comprehension:}
$$(\forall z: A)(z \rin \{x:A \; |\; \vf\} \Leftrightarrow \vf[z/x]).$$
\end{itemize}

The extensionality axiom gives the following
\begin{lemma} \label{elevel_lm}
 For any type symbol $A$, there is a formula $\varphi_A(x,y)$ 
in $\formulas{\elevel{A}}$ such that
$$x=_A y \Longleftrightarrow \varphi_A(x,y).$$
\end{lemma}
\begin{proof} Let $\varphi_A(x,y) \equiv (x=_Ay)$ for $A=\unittype, \nattype$,
and $\varphi_{B \times C}(x,y) \equiv (\pi_1(x)=_B \pi_1(y) \land 
\pi_2(x)=_C \pi_2(y))$. Further, define
$$\varphi_{\rpw_n(A)}(x,y) \equiv (\forall z:A)(z \rin x \Leftrightarrow z \rin y).$$
The righthand side formula has level $\max(n, |A|) = || \rpw_n(A) ||$, and is by the extensionality axiom 
equivalent to $x=_{\rpw_n(A)} y$. It is now easy prove the properties of the other
types by induction. 
\end{proof}

\medskip
To state the Functional Reducibility Axiom, which is the final axiom, we need to introduce some terminology. Inspired by the terminology in (Bell 1988) of the syntactic counterpart to toposes as {\em local set theories,}
we define a {\em local set} to be a 
type  $A$ together with an element $X$ of $\rpw_n(A)$, for some $n$. It is thus specified by a triple $(A,X,n)$, where 
$A$ is the underlying type, $X$ is the propositional function defining the subset of $A$ and $n$ the level
of the propositional function. A basic example is  the natural numbers as a local set  
given by 
$${\mathbb N} = (\nattype, \{ x: \nattype \; | \; \top \}, 0).$$
A {\em relation} $R$ between $(A,X,n)$ and $(B,Y,m)$ is some $R: \rpw_{k}(A \times B)$ such that
$$(\forall x:A)(\forall y:B)(\pair{x}{y} \rin R \Longrightarrow x \rin X \land y\rin Y).$$
Such a relation is {\em functional}  if
$$(\forall x:A)(\forall y, z:B)(\pair{x}{y} \rin R \land \pair{x}{z} \rin R \Longrightarrow y =_B z),$$
and is {\em total} if
$$(\forall x:A)(x \rin X \Rightarrow (\exists y: B)(y \rin Y \land \pair{x}{y} \rin R)).$$
A functional, total relation is simply called a {\em map}. 

\medskip
Now the central axiom is the following:

{\em Functional Reducibility Axiom:} For $A,B \in \rsymb$, $m,n \in {\mathbb N}$, we have for 
$k= \elevel{B} \inmax m \inmax n$ that for any $r \in {\mathbb N}$
\begin{eqnarray*}
\lefteqn{(\forall X : \rpw_m(A))(\forall Y: \rpw_n(B))(\forall F: \rpw_{r}(A \times B))} \\
&& \qquad \Bigl[\mbox{$F$ map from $(A,X,m)$ to $(B,Y,n)$} 
\Rightarrow \\
&& \qquad \qquad \qquad\quad (\exists G:  \rpw_{k}(A \times B))(\forall z:A \times B)(z \rin F \Leftrightarrow z \rin G)\Bigr]
\end{eqnarray*}

Note: $G$ is necessarily unique by the extensionality axiom.

\medskip
These are the axioms of the basic theory IRTT. To make the system useful for developing Bishop style constructive analysis, we may also extend it by the {\em Relativized Dependent Choice} (RDC) axiom scheme, which is  the following: 

\medskip
{\em RDC:} Let $A$ be any sort and $m,n \ge 0$. Then we have the axiom: for any $D : \rpw_m(A)$, any 
$R : \rpw_n(A \times A)$, and any $a:A$ satisfying
$$ a \rin D \land (\forall x:A)(x \rin D \Rightarrow (\exists y:A)(y \rin D \land \langle x,y \rangle \rin R))$$
there is $F: \rpw_{k}(\nattype \times A)$ a map from ${\mathbb N}$ to $(A,D,m)$, satisfying
\begin{itemize}
\item[(a)] $\langle 0,a \rangle \rin F$,
\item[(b)] $(\forall i:\nattype)(\forall y,z: A)(\langle i,y \rangle \rin F \land 
\langle i+1,z \rangle \rin F  \Rightarrow \langle y,z \rangle \rin R)$.
\end{itemize}
Here $k= \rlevel{A}$.

\begin{remark} \label{russelltypes}
{\em The ramified types of Russell are --- according to the modern
elaboration of Laan and Nederpelt (1996) --- given by the following inductive
definition. Each ramified type has the form $t^n$ where $n$ is a natural number
indicating the order of the type. These are generated as
follows
\begin{itemize}
\item[(a)] $0^0$ is a ramified type (the type of individuals)
\item[(b)] if $t_1^{n_1}, \ldots, t_k^{n_k}$ are ramified types, and $m> n_1,\ldots, n_k$, then
$(t_1^{n_1}, \ldots, t_k^{n_k})^m$ is a ramified type.
\end{itemize}
A ramified type is {\em minimal} (or {\em predicative}) if in each application of (b) in its construction, one takes
$m=1+\max(n_1,\ldots,n_k)$. The reducibility axiom then states that an element of 
a type $t^n$ is extensional equivalent  to some element of a corresponding minimal
type (Kamareddine {\em et al.} 2002, p.\ 233).

These ramified types can be interpreted into the types  $\rsymb$ as follows, assuming the type individuals is interpreted as the type of natural numbers:
\begin{eqnarray*}
\rint{0^0} & = & \nattype \\
\rint{(t_1^{n_1}, \ldots, t_k^{n_k})^m}  & = & 
\rpw_m(\rint{t_1^{n_1}} \times \cdots \times \rint{t_k^{n_k}})\\
\end{eqnarray*}
Here $A_1 \times \cdots \times A_k = ( \cdots (A_1 \times A_2) \times \cdots )\times A_k$, which in case $k=0$ is just the unit type $\unittype$.  The types of $\rsymb$ are thus richer than Russell's,
but still have a predicative interpretation as is demonstrated in the following sections.

}
\end{remark}

%   SETOIDS
%
%
%
%

\section{Setoids}
As interpreting theory we consider Martin-L{\"o}f type theory with an
infinite sequence of universes ${\rm U}_0, {\rm U}_1, {\rm U}_2,
\ldots$. Each universe ${\rm U}_n$ is closed under the standard type
constructions $\Pi$, $\Sigma$, $+$ and ${\rm Id}$. ${\rm U}_0$ contains basic
types such as the type $N$ of natural numbers, empty type and unit
type. Moreover if $A:{\rm U}_n$ then $A$ is a type and $A: {\rm
U}_{n+1}$. Finally ${\rm U}_n: {\rm U}_{n+1}$. This is as presented in
(Martin-L{\"o}f 1984) although we assume that the identity type ${\rm Id}$ is
intensional instead of extensional. This theory is considered predicative
in the strict sense of Feferman and Sch{\"u}tte and its proof-theoretic
ordinal is $\Gamma_0$ (Feferman 1982)
On the propositions-as-types interpretation, the universe $U_n$ can be regarded as the
type of propositions of level $n$.  

A setoid {\em $A=(\typeof{A},\eqof{A})$ is of index $(m,n)$} if
$\typeof{A}: {\rm U}_m$ and $\eqof{A}: \typeof{A} \to \typeof{A} \to
{\rm U}_n$. We also say that $A$ is an {\em $(m,n)$-setoid.} 
Since any type of ${\rm U}_k$ is a type of ${\rm U}_{k'}$ for any
$k' \ge k$, it follows that any $(m,n)$-setoid is an $(m',n')$-setoid
whenever $m' \ge m$ and $n' \ge n$.

 $A$ is
said to be an {\em $n$-setoid} if it is an $(n,n)$-setoid. It is an
{\em $n$-classoid} if it is an $(n+1,n)$-setoid.

\begin{examples} 

{\em
\begin{itemize}
\item[]
\item[(a)] $(N,{\rm Id}(N,\cdot,\cdot))$ is a 0-setoid.
\item[(b)] Aczel's model of CZF $(V,=_V)$ is a 0-classoid if $V$ is built from
the universe $U_0$.
\item[(c)]  $\Omega_n = (U_n, \leftrightarrow)$ is an $n$-classoid, when
$\leftrightarrow$ is logical equivalence.
\end{itemize}}
\end{examples}

For setoids $A$ and $B$, the product construction $A \times B$ is provided
by 
$$|A \times B| = |A| \times |B|$$
and
$$(x,y) =_{A \times B} (u,v) \Longleftrightarrow (x=_A u) \land (y=_B v).$$
Whereas the exponent construction $B^A$ is given
by
$$|B^A| =_{\rm def} (\Sigma f: |A| \to |B|)(\forall x,y: |A|)(x=_A y \Rightarrow f(x) =_B f(y))$$
and
$$(f,p) =_{B^A} (g,q) \Longleftrightarrow_{\rm def} (\forall x: |A|)(f(x)=_B g(x))$$
Thus an element $g=(|g|,{\rm ext}_g)$ of $|B^A|$ consists of a
function $|g|$ and a proof ${\rm ext}_g$ of its extensionality.

\begin{lemma} \label{exprule}
If $A$ is an $(m,n)$-setoid and $B$ is a $(k,\ell)$-setoid then
$A \times B$ is a $(m \inmax k, n \inmax \ell)$-setoid and 
$B^A$ is a $(m \inmax n \inmax k \inmax \ell,  m \inmax \ell)$-setoid. $\qed$
\end{lemma}

\medskip
To simplify notation in the sequel, we usually write $x:A$ for $x:|A|$
when $A$ is a setoid. Moreover, we write $g(x)$ for $|g|(x)$ when $g:[A
  \to B]$.
We have the principle of unique choice that gives a
one-to-one correspondence between functions and total, functional relations:

\begin{thm} Let $A$ and $B$ be setoids.
Suppose that $R(x,y)$ is an extensional property depending on $x:A$,
$y:B$.
$$(\forall x:A)(\exists ! y:B)R(x,y) \Longrightarrow 
(\exists !f:B^A)(\forall x:A)R(x,f(x)). \; \qed$$
\end{thm}

\medskip
For any setoid $A$, note that an element 
$R=(|R|, {\rm ext}_R) : [A \to \Omega_n]$ 
consists of a predicate $|R|: |A| \to U_n$ and 
proof of its extensionality,  i.e.\ that if for all $x,y: |A|$
$$x=_A y \Longrightarrow (R(x) \Leftrightarrow R(y)).$$
We call this an {\em extensional propositional function on $A$ of level $n$.}
The set $[A \to \Omega_n]$ will be the interpretation of $\rpw_n(A)$.

We finally recall that the following form of dependent choice is valid for setoids.

\begin{thm} \label{ttrdc}
Let $A$ be a setoid and let $R$ be an extensional predicate on $A \times A$.
Suppose that $(\forall x:A)(\exists y:A)R(\langle x,y \rangle)$.
Then for any $x : A$ there is $f: [{\mathbb N} \to A]$ so that $f(0)=_A x$ and for all $n: {\mathbb N}$: 
$$R(\langle f(n),f(n+1) \rangle). \; \qed$$
\end{thm}

%   MODEL
%
%
%
%

\section{A model of ramified type theory}

The type symbols of $\rsymb$ interpret naturally as an extensional hierarchy of setoids in the background theory.
Define setoids $\ipret{S}$ by recursion on the structure of $S \in \rsymb$.
\begin{eqnarray*}
\ipret{{\bf 1}} &= & ({\rm N}_1, {\rm Id}({\rm N}_1,\cdot,\cdot)) \\
\ipret{{\bf N}} & = & ({\rm N}, {\rm Id}({\rm N},\cdot,\cdot)) \\
\ipret{(S \times T)} & = &\ipret{S} \times \ipret{T} \\
\ipret{\rpw_k(S) } & =  & [{\ipret{S}} \to {\Omega_k}].  \\
\end{eqnarray*}

\begin{lemma} \label{lvlemma}
 If $S \in \rsymb$ , then $\ipret{S}$ is an $(\rlevel{S},\elevel{S})$-setoid
 and $\rlevel{S} \ge \elevel{S}$.
\end{lemma}
\begin{proof}  By induction on the structure of $S$. If  $S$ is ${\bf 1}$ or ${\bf N}$, then $\ipret{S}$
is an $(0,0)$-setoid,  and $\rlevel{S}=\elevel{S}=0$. If $S= S_1 \times S_2$, then
 $\rlevel{S}= \rlevel{S_1} \inmax \rlevel{S_2}$ and $\elevel{S}= \elevel{S_1} \inmax \elevel{S_2}$. Now $\ipret{S} = \ipret{S_1} \times \ipret{S_2}$. By Lemma  
 \ref{exprule} , and the induction hypothesis, we have that $\ipret{S_1} \times \ipret{S_2}$,
 is a $(\rlevel{S_1} \inmax \rlevel{S_2}, \elevel{S_1} \inmax \elevel{S_2})$-setoid,
 and thus by definition of levels, an $(\rlevel{S}, \elevel{S})$-setoid.  By inductive hypothesis also $\rlevel{S} \ge \elevel{S}$. For
 the case $\rlevel{S}= \rpw_n(A)$, we have that 
 $$\ipret{S} =[{\ipret{A}} \to \Omega_n].$$
Now by inductive hypothesis $\ipret{A}$ is an $(\rlevel{A},\elevel{A})$-setoid.
Also $\Omega_n$ is an $(n+1,n)$-setoid. Thus by Lemma \ref{exprule},  $\ipret{S}$
is a setoid of index
 $$(\rlevel{A} \inmax \elevel{A}   \inmax  (n+1) \inmax n, \rlevel{A} \inmax n) = 
 ((n+1) \inmax \rlevel{A}, n \inmax \rlevel{A} ) = (\rlevel{S}, \elevel{S}).$$
 Here we have used the inductive hypothesis $\rlevel{A} \ge \rlevel{A}$. 
 Clearly we have $\rlevel{S} \ge \elevel{S}$. 
 \end{proof}

The interpretation $\ipret{(-)}$ is now extended in the standard fashion
for propositions-as-types interpretations of many-sorted
intuitionistic logic (cf.\ Martin-L{\"o}f 1998). Each formula $\vf$ is
interpreted
as a type $\ipret{\vf}$. Each term $a$ of sort $A$ is interpreted as 
an element $\ipret{a}$ of type $|\ipret{A}|$. Moreover,  if $\vf$ is in 
$\formulas{n}$ we shall require that $\ipret{\vf} : {\rm U}_n$.

Each variable $x$ of sort $A$ is interpreted as a variable $\ipret{x}$ of type $|\ipret{A}|$.
All the terms associated with the sorts $\unittype$, $\nattype$, $A \times B$ are interpreted
in the obvious way.  Logical constants are interpreted as the corresponding type constructions
in the familiar way.
E.g.\ for quantifiers we define
$$\ipret{((\forall x:A)\vf)} = (\Sigma \ipret{x} : |\ipret{A}|)\ipret{\vf} \qquad
\ipret{((\exists x:A)\vf)} = (\Pi \ipret{x} : |\ipret{A}|)\ipret{\vf}
$$
Then for atomic formulas define
$$ \ipret{(a =_A b)} = (\ipret{a} =_{\ipret{A}} \ipret{b}) \qquad 
\ipret{(c \rin d)} = |\ipret{d}|(\ipret{c}) 
 $$
where $a$ and $b$ are terms of sort $A$, and $c$ is a term of sort $B$ and $d$ a term of sort $\rpw_k(B)$.
If $\vf \in \formulas{n}$ and $x$ is variable of sort $A$, then define $\ipret{\{ x: A \; | \; \vf \}} = 
(\lambda \ipret{x}. \ipret{\vf}, e)$ where $e$ is a proof object for the extensionality of 
$\lambda \ipret{x}. \ipret{\vf}: |\ipret{A}| \to {\rm U}_n$.

\begin{lemma} For $\vf \in \formulas{n}$, the interpretation satisfies $\ipret{\vf} : {\rm U}_n$.
\end{lemma}
\begin{proof} By induction on the build-up of formulas. We do some interesting cases: 

$\vf = (a =_A b) $:  if $\rlevel{A} \le n$, then since $\ipret{A}$ is an $n$-setoid, 
$\ipret{\vf} : {\rm U}_n$.

$\vf = (c \rin d) $: if $d$ has sort $\rpw_k(A)$ where $k \le n$, then  $\ipret{\vf} = |\ipret{d}|(\ipret{c}) : {\rm U}_k$.
But $k \le n$, so indeed $\ipret{\vf} : U_n$.

$\vf = (\forall x:A)\psi$: if  $\rlevel{A} < n$, $\psi \in \formulas{n}$, then $\ipret{\vf} = 
(\Pi \ipret{x}: |\ipret{A}|)\ipret{\psi} \in {\rm U}_n $, since $\ipret{A}$ is an $n$-setoid and
$\ipret{\psi} \in {\rm U}_n$ according to the inductive hypothesis.

\end{proof}

\medskip

Next we consider the semantic version of a local set.
A pair $M = (S_M,\chi_M)$ consisting of $S_M$ an $(m,n)$-setoid and a propositional
function  $\chi_M \in [S_M \to \Omega_k]$ is called a {\em local set.} It
gives rise to a setoid
$$\hat{M}= ((\Sigma x: S_M)\chi_M(x), =')$$
where 
$$(x,p) =' (y,q) \Longleftrightarrow_{\rm def} x=_{S_M} y.$$
This setoid has index $(m \inmax k,n)$.

For two local sets $M$ and $K$, we call $F\in [S_M \times S_K \to
  \Omega_r]$ a {\em map from $M$ to $K$} if
$$(\forall x:S_M)(\forall y:S_K)(F(\pair{x}{y}) \Longrightarrow \chi_M(x) 
\land \chi_K(y))$$
and 
$$(\forall x:S_M)(\chi_M(x) \Rightarrow 
(\exists ! y: S_K)(\chi_K(y)  \land F(\pair{x}{y}))).$$

The following theorem implies the validity of the Functional Reducibility
Axiom under the interpretation. 

\begin{thm}[Functional reducibility] \label{funred} Let $A$ be an $(m,n)$-setoid 
and $B$ be an $(m',n')$-setoid. Suppose $X : [A \to \Omega_k]$
and $Y : [B \to \Omega_{k'}]$. Let $r \in {\mathbb N}$. 
Then for every map $F: [A\times B \to \Omega_r]$ from $(A,X)$ to
$(B,Y)$ 
 there is 
$G :[A\times B \to \Omega_{\ell}]$  such that for all $a:A$ and
$b: B$
$$F(\pair{a}{b})\Leftrightarrow G(\pair{a}{b}).$$
Here $\ell= k \inmax k' \inmax n'$.
\end{thm}
\begin{proof} Construct setoids $U=\widehat{(A,X)}$ and 
$V=\widehat{(B,Y)}$.
These are  $(m \inmax k,n)$- and  $(m' \inmax k',n')$-setoids, respectively. 
% The exponent $[U\to V]$ is thus a $(\max(m,n,k,m',n',k'),\max(m,n',k))$-setoid. 
For $f : [U\to V]$ define the
relation 
$$G_f(\pair{a}{b}) \equiv (\exists p:X(a))(\exists q:Y(b))(f(\langle a,p \rangle) =_V \langle b,q \rangle). $$
Since $x=_Vy$ is in $U_{n'}$,
we have that the righthand side is in $\Omega_{\ell}$ where $\ell= k \inmax k' \inmax n'$.
Suppose now that $r \in {\mathbb N}$ and that $F: [A\times B \to \Omega_r]$ is a map from $(A,X)$ to
$(B,Y)$. This implies that
$$(\forall x:A)(X(x) \Rightarrow 
(\exists ! y: B)(Y(y)  \land F(\pair{x}{y})))$$
Thus
$$(\forall u: U)(\exists! v:V)F(\pair{\pi_1(u)}{\pi_1(v)}).$$
By the axiom of unique choice there is a unique $f : [U\to V]$ so that
\begin{equation} \label{chosen}
(\forall u: U)F(\pair{\pi_1(u)}{\pi_1(f(u))}).
\end{equation}
Suppose that $F(\pair{a}{b})$ for $a: A$ and $b:B$. By strictness of $F$ there is $p:X(a)$.
Let $u= \langle a, p \rangle$. By (\ref{chosen}) we have $F(\pair{a}{\pi_1(f(u))})$. 
Thus $\pi_1(f(u)) =_B b$. Let $q= \pi_2(f(u))$. Thus $f(\langle a,p \rangle) =_V \langle b, q \rangle$,
i.e.\ $G_f(\pair{a}{b})$. Conversely, suppose $G_f(\pair{a}{b})$, so that 
$f(\langle a,p \rangle) =_V \langle b, q \rangle$ for some $p:X(a)$ and $q:Y(b)$. By (\ref{chosen}) 
we have $F(\pair{\pi_1(\langle a,p \rangle)}{\pi_1(f(\langle a,p \rangle))})$. Thus since $F$ is 
extensional, $F(\pair{a}{b})$. This proves $$F(\pair{a}{b}) \Leftrightarrow G_f(\pair{a}{b}).$$
\end{proof}

We verify the Functional Reducibility Axiom. Suppose  $A,B \in \rsymb$, $m,n \in {\mathbb N}$ and
let $k= \elevel{B} \inmax m \inmax n$. By Lemma \ref{lvlemma} $\ipret{B}$, 
is a $(\rlevel{B},\elevel{B})$-setoid.
 Suppose $\ipret{X} : \ipret{\rpw_m(A)}$,  $\ipret{Y} : \ipret{\rpw_n(B)}$,
$\ipret{F} : \ipret{\rpw_r(A\times B)}$ and that
$$\ipret{(\mbox{$F$ is a map from $(A,X,n)$ to $(B,Y,m)$})}$$
is true. This says that $\ipret{F}$ is a map from $(\ipret{A},\ipret{X})$ to $(\ipret{B},\ipret{Y})$. 
By Theorem \ref{funred} above we get for $\ell= m \inmax n \inmax \elevel{B}$ some
 $G :[\ipret{A}\times \ipret{B} \to \Omega_{\ell}]$  such that for all $a:\ipret{A}$ and
$b: \ipret{B}$
$$\ipret{F}(\pair{a}{b})\Leftrightarrow G(\pair{a}{b}).$$

The axiom RDC is verified  as follows. Let $A \in \rsymb$  and $m,n \ge 0$. Let $k= \rlevel{A}$. Suppose $\ipret{D} : \ipret{\rpw_m(A)}$,
$\ipret{R}: \ipret{\rpw_n(A \times A)}$ and $x: \ipret{A}$ are such that $\ipret{D}(x)$ and
\begin{equation} \label{totop}
\ipret{\bigl((\forall x:A)(x \rin D \Rightarrow (\exists y:A)(y \rin D \land \langle x,y \rangle \rin R))\bigr)}
\end{equation}
are true.
Setting $E = (\ipret{A}, \ipret{D})$ the assumption (\ref{totop}) implies that
$$(\forall u:\hat{E})(\exists v:\hat{E})\ipret{R}(\langle\pi_1(u),\pi_1(v)\rangle).$$
We have $t:\ipret{D}(x)$ for some $t$, so $w=\langle x, t \rangle: \hat{E}$. By Theorem \ref{ttrdc}
there is $f: [{\mathbb N} \to E]$ so that $f(0)=_E w$ and for all $i: {\mathbb N}$
$$\ipret{R}(\langle\pi_1(f(i)), \pi_1(f(i+1))\rangle).$$
Define 
$$F(\langle i, a \rangle) = (\pi_1(f(i))=_A a)$$
Now $F(\langle i, a\rangle)$ is in ${\rm U_k}$, so it is easy to see that $F$ satisfies the requirements.

We thus conclude:

\begin{thm} IRTT + RDC can be interpreted in Martin-L{\"o}f type theory with an infinite sequence of universes. $\qed$
 \end{thm}

\section{Constructions using local sets in {\bf IRTT}}

The system IRTT is not primarily intended to be practical for formalization, 
but  a theoretical exhibit
to clarify the relation between Russell's type theory and modern type theories. IRTT can straightforwardly be embedded into modern proof assistants based on Martin-L\"of type theory. The notation of IRTT is undeniably quite cumbersome to handle 
because of the levels associated to types, a property it inherits from Russell's system. To simplify
its use we can formulate the abstract properties of the local sets in category-theoretic terms
as in (Bell 1988).  

\medskip
\subsection*{The category of local sets}
%We consider two local sets $(A, X, m)$ and $(B, Y, n)$ in IRTT 
% as {\em extensionally equal} if $A=B$ (i.e.\ the type symbols are the same)
% and
% $$\{x:A \;|\; x \rin X\}=_{\rpw_{\max(m,n)}(A)}\{x:A \;|\; x \rin Y\}.$$
% By extensionality the latter equality holds if, and only if,
% $$(\forall x:A)(x \rin X \Longleftrightarrow x \rin Y).$$
% Note that if $m=n$, this is actually equivalent to $X=_{\rpw_m(A)} Y$.
The objects of the category are the local sets of IRTT.
A morphism between locals set ${\bf X}=(A, X, m)$ and  ${\bf Y}=(B, Y, n)$ is a pair
${\bf F}=(F, k)$ such that $F:\rpw_k(A\times B)$ is a map from ${\bf X}$ to 
${\bf Y}$. We write ${\bf F}: {\bf X} \to {\bf Y}$. Two such morphisms
$(F,k)$ and $(F',k')$ are equal if they have extensionally equal graphs:
$$(\forall x:A)(\forall y:B)((x, y) \rin F \Longleftrightarrow (x, y) \rin F').$$
Suppose that ${\bf F}: {\bf X} \to {\bf Y}$ and 
${\bf G}= (G,\ell): {\bf Y} \to {\bf Z}= (C, Z, p)$ are morphisms. The composition ${\bf G} \circ {\bf F}= (G \circ F, q)$ is then given by 
$$G \circ F =_{\rm def} \{ w: A \times C \;|\; 
(\exists y: B)(\langle \pi_1(w), y\rangle \rin F \land \langle y,
\pi_2(w) \rangle \rin G) \}$$
and since
$G \circ F : \rpw_{\rlevel{B} \inmax k \inmax \ell}(A \times C)$, we can take 
$q=\rlevel{B} \inmax k \inmax \ell$.
Hence
$${\bf G} \circ {\bf F} :  (A, X, m) \to  (C, Z, p).$$
For a local set ${\bf X}=(A, X, m)$ its identity map $1_{\bf X} : {\bf X} \to {\bf X}$ is given
by $1_{\bf X} = (1_X, m \inmax \elevel{A})$ where
$$1_X = \{ w: A \times A :  \pi_1(w) \rin X \land \pi_2(w) \rin X \land \pi_1(w) =_A \pi_2(w) \}$$
It is easily verified that the local sets form a category using extensional equality 
of maps. 

\subsection*{Real numbers}
Note that by the Functional Reducibility Axiom, every $$(F,k): (A, X, m) \to (B, Y, n)$$ is
extensionally equal to some $(F',\elevel{B} \inmax m \inmax n): (A, X, m) \to (B, Y, n)$, so hom-sets are
"small", in the sense that they are limited by the level of the domains and codomains.
In particular, for  $(A, X, m) = (B, Y, m) = {\mathbb N} = (\nattype, \{ x: \nattype \; | \; \top \}, 0)$ and any  $(F,k)$ as above is equal to some $(F',0) : {\mathbb N} \to {\mathbb N}$.  By this we may conclude that (Cauchy) real numbers of
IRTT + RDC all live on the first level of functions.

\subsection*{Quotient sets}
The  quotient sets construction is similar to that in
simple type theory or topos logic, but we need to keep track of
the levels of the power sets involved.
Let ${\bf X} = (A, X, m)$ be a local set and suppose ${\bf E}Ê= (A \times A, E, n)$ is a local set that represents
an equivalence relation on ${\bf X}$, i.e.\  it satisfies
\begin{eqnarray*}
&& x \rin X \Leftrightarrow \langle x,x \rangle \rin E \\
&& \langle x, y \rangle \rin E \Rightarrow \langle y,x \rangle \rin E \\
&& \langle x, y \rangle \rin E \land \langle y, z \rangle \rin E  \Rightarrow \langle x,z \rangle \rin E 
\end{eqnarray*}
Define the quotient $X/E : \rpw_{\ell}(\rpw_{\ell}(A))$ by
$$X/E = \{ S:\rpw_{\ell}(A) \; | \; (\exists x: A)(x\rin X \land (\forall y:A)(y \rin S \Leftrightarrow \langle x,y \rangle \rin E))\}$$
Here $\ell = m \inmax n \inmax \rlevel{A}$. Write $B=
\rpw_{\ell}(A)$. Thus we have a local set $ {\bf X}/{\bf E} =(B, X/E, \ell)$, that we shall see is the quotient set of ${\bf X}$ by ${\bf E}$. Define the quotient map $(Q, \ell)$ from
${\bf X}$ to  ${\bf X}/{\bf E}$ by
$$Q= \{w: A\times B \; |\; \pi_1(w) \rin X \land \pi_2(w) \rin X/E
\land \pi_1(w) \rin \pi_2(w) \}.$$
We show that $(Q, \ell)$ satisfies the universal property for a
quotient. Suppose that $(F, k): {\bf X} \to {\bf Z}=(C, Z, p)$ is a map that respects the equivalence
relation  $E$, i.e.\
\begin{equation} \label{eresp}
\langle x,z \rangle \rin F \land 
\langle x',z' \rangle \rin F \land \langle x,x'\rangle \rin E
\Longrightarrow z =_C z'.
\end{equation}
Define a map ${\bf H}=(H, r): {\bf X}/{\bf E} \to {\bf Z}$ by 
$$H =\{w : B\times C \; | \; \pi_1(w) \rin X/E \land 
(\exists u:A)(u\rin X \land u \rin \pi_1(w) \land 
\langle u, \pi_2(w) \rangle \rin F)\}.$$
Here $r= \ell \inmax k \inmax p$.
We have  that $\langle x,z\rangle \rin (H \circ Q)$ is equivalent to
$$x \rin X\land (\exists S:B)(S\rin X/E \land x \rin S \land 
(\exists u:A)(u\rin X \land u \rin S \land \langle u,z\rangle \rin F)).
$$
From this and (\ref{eresp}) it follows that $\langle x,z\rangle \rin
F$. Conversely, suppose $\langle x,z\rangle \rin
F$. For $S:B$ we may take $S= \{y: A \;|\; \langle x,y \rangle \rin
E\}$,
as $n \le \ell$. It follows immediately that $\langle x,z\rangle \rin
(H \circ Q)$. Thus $F$ is extensionally equal to $H \circ Q$.

Now suppose that ${\bf H'}=(H', s) : {\bf X}/{\bf E} \to {\bf Z}$ is another map such that $H' \circ Q$ is extensionally equal to $F$. We show
that $H'$ is extensionally equal to $H$: 
suppose $\langle S,z \rangle \in H'$. 
Then $S \rin X/E$, so there is some $x \rin X$ with $x \rin S$. Hence $\langle x,S \rangle \rin Q$,
and so $\langle x,z \rangle \rin (H' \circ Q)$. Thus  $\langle x,z \rangle \rin F$. By definition of
$H$ we get $\langle S, z \rangle \rin H$. For the converse, suppose $\langle S,z \rangle \rin H$.
Then there is some $ u \rin  X$ with $u \rin S$ and $\langle u,z \rangle \rin F$. Thus by assumption
$\langle u,z \rangle \rin  (H' \circ Q)$. Then there is some $T$ with $\langle u, T \rangle \in Q$ and
$\langle T, z \rangle \rin H'$. Then $u \rin T$, so in fact we have $S=T$ as $T \rin X/E$. Hence 
$\langle S, z \rangle \rin H'$ as required.

\subsection*{Products} For two local sets  ${\bf X}_1 =(A_1, X_1, m_1)$ and
${\bf X}_2 =(A_2, X_2, m_2)$, define their {\em binary product} to be
$$X_1\times X_2 =\{z : A_1\times A_2 \;|\; \pi_1(z) \rin X_1 \land \pi_2(z) \rin
X_2\}:
\rpw_{m_1 \inmax m_2}(A_1\times A_2).$$
This gives a local set ${\bf X}_1 \times {\bf X}_2 =
(A_1 \times A_2, X_1 \times X_2, m_1 \inmax m_2)$.
The projection maps are  
\begin{eqnarray*}
p_i &=& \{z: (A_1\times A_2) \times A_i \;|\; 
\pi_1(z) \rin X_1\times X_2 \land 
\pi_2(z)\rin X_i \land
\pi_i(\pi_1(z)) =_{A_i} \pi_2(z) 
 \} 
\end{eqnarray*}
for $i=1,2$.
Then $p_i: \rpw_{r_i}((A_1\times A_2) \times A_i)$, where
$r_i = m_1 \inmax m_2 \inmax \elevel{A_i}$,
and $$(p_i, r_i) :  {\bf X}_1 \times {\bf X}_2 \to {\bf X}_i.$$
For maps ${\bf F}=(F,p) : {\bf Z} \to {\bf X}_1$
and ${\bf G}=(G, q): {\bf Z} \to {\bf X}_2$, where  ${\bf Z} = (C, Z, k)$,
define 
$$\langle F, G\rangle = 
\{w: C \times (A_1 \times A_2) \;|\; \langle \pi_1(w),\pi_1(\pi_2(w))
  \rangle \rin F \land 
\langle \pi_1(w), \pi_2(\pi_2(w)) \rangle \rin G
  \}.$$
We have $\langle {\bf F}, {\bf G}\rangle =(\langle F,G\rangle, p \inmax q) : {\bf Z}  \to  {\bf X}_1 \times {\bf X}_2 $.
It is straightforward to check that these constructions make up
a category-theoretic product of the local sets  ${\bf X}_1$ and ${\bf X}_2$.

The {\em terminal object} (0-ary product) is given by the local set 
$\tilde{\bf 1} =({\bf 1}, \{ x: {\bf 1} \;|\; \top \}, 0)$.

\subsection*{Exponential sets}
Again this construction is similar to the construction in simple type theory.
Let ${\bf X} = (A, X, m)$ and ${\bf Y} = (B, Y, n)$ be local sets  and let $k= \elevel{B} \inmax m \inmax n$ be as in
the Functional Reducibility Axiom.  Define
$$Y^X = \{F: \rpw_k(A\times B)\; | \; \mbox{$F$ is a map from ${\bf X}$ to ${\bf Y}$} \}.$$
Then by examining the type level involved in the definition of amap one sees that
$Y^X: \rpw_s(\rpw_k(A \times B))$ where 
$s= \rlevel{A} \inmax  \elevel{A} \inmax \rlevel{B} \inmax \elevel{B} \inmax m \inmax n =
\rlevel{A} \inmax \rlevel{B} \inmax m \inmax n$, giving the
local set $${\bf Y}^{\bf X}=(\rpw_k(A \times B), Y^X, s).$$
We have that the product 
$${\bf Y}^{\bf X} \times {\bf X} = (\rpw_k(A\times B) \times A, Y^X \times X, s \inmax n)
= (\rpw_k(A\times B) \times A, Y^X \times X, s).$$
The evaluation map ${\bf ev} = ({\rm ev},s) : {\bf Y}^{\bf X} \times {\bf X}  \to {\bf Y}$ is given by
\begin{eqnarray*}
\lefteqn{
{\rm ev } = \{ w: (\rpw_k(A\times B) \times A) \times B \; |\; 
\pi_1(w) \rin Y^X \times X, \pi_2(w) \rin Y,} \\
&& \qquad \qquad \qquad \qquad \qquad  \qquad \qquad \langle\pi_2(\pi_1(w)), \pi_2(w) \rangle \rin \pi_1(\pi_1(w)) \}.
\end{eqnarray*}
Consider an arbitrary local set ${\bf Z} = (C, Z, p)$ and an arbitrary map 
${\bf G}=(G, q)$ from ${\bf Z} \times {\bf X}$ 
to ${\bf Y}$.
Define a map ${\bf H} = (H, k):  {\bf Z} \to {\bf Y}^{\bf X}$ by
\begin{eqnarray*}
\lefteqn{ H=\{w : C \times \rpw_k(A \times B) \; |\; \pi_1(w) \rin Z, \pi_2(w) \rin Y^X,  } \\
&&\qquad \qquad \qquad \qquad \qquad (\forall x:A)(\forall y:B)(\langle x,y \rangle \rin \pi_2(w) \Leftrightarrow \langle \langle \pi_1(w),x \rangle, y \rangle \rin G)
\}.
\end{eqnarray*}
Here $t=  \rlevel{A} \inmax \rlevel{B} \inmax \rlevel{C} \inmax m \inmax n \inmax p \inmax q$.
For $z: C$, $x:A$ and $y:B$ we have 
$$\langle \langle z, x\rangle, y\rangle \rin {\rm ev } \circ (H \times {\rm id}) =
 {\rm ev } \circ \langle H \circ p_1, p_2\rangle$$
if, and only if, there are $S: \rpw_k(A \times B)$ and $u:A$ with 
$$\langle \langle z, x \rangle, \langle S, u \rangle \rangle \rin  \langle H \circ p_1, p_2\rangle \land 
\langle \langle S, u \rangle,y \rangle \rin  {\rm ev },$$
 which is equivalent to
$$\mbox{$\langle z, S \rangle \rin H $ and  $
\langle \langle S, x \rangle, y \rangle \rin  {\rm ev }$}$$
and spelling this out we get
$$\mbox{$z \rin Z$, $S \rin Y^X$, $y \rin Y$, $\langle x,y \rangle \rin S$ and $(\forall x:A)(\forall y:B)(\langle x, y \rangle
 \rin S \Leftrightarrow 
\langle \langle z,x \rangle, y \rangle \rin G)$}$$
Clearly this implies  $\langle \langle z, x\rangle, y\rangle \rin G$. Thus we have shown the implication
$$\langle \langle z, x\rangle, y\rangle \rin {\rm ev } \circ (H \times {\rm id})  \Longrightarrow
\langle \langle z, x\rangle, y\rangle \rin G.$$
To show the converse, we assume $\langle \langle z, x\rangle, y\rangle \rin G$. Let
$$L = \{ w : A \times B\; | \; \langle\langle z, \pi_1(w) \rangle, \pi_2(w) \rangle \rin G \} : \rpw_q(A \times B).$$
This is a map from ${\bf X}$ to ${\bf Y}$. By the Functional Reducibility Axiom there is $S: \rpw_k(A \times B)$
so that 
$$(\forall u:A)(\forall v:B)(\langle u,v \rangle \rin S \Longleftrightarrow 
\langle u,v \rangle \rin L).$$
But $\langle x,y \rangle \rin L \Leftrightarrow \langle \langle z,x \rangle, y \rangle \rin G$, and  hence $\langle x,y \rangle \rin S$, 
so we have indeed $\langle \langle z, x\rangle, y\rangle \rin {\rm ev } \circ (H \times {\rm id})$. Thus ${\rm ev } \circ (H \times {\rm id})$ and $G$ are extensionally equal.

Finally, we check uniqueness of $H$. Suppose $H'$ is another map from 
${\bf Z}$ to ${\bf Y}^{\bf X}$ such that 
${\rm ev } \circ (H' \times {\rm id})$ and $G$ are extensionally equal. Thus we have 
\begin{equation} \label{ekv1}
\langle\langle z,x\rangle, y \rangle \rin G \Leftrightarrow (\exists S'': \rpw_k(A \times B))(\langle z,S'' \rangle 
\rin H'  \land \langle x,y \rangle \rin S'')
\end{equation}
We check that $H$ and $H'$ are equal maps. Suppose $\langle z, S \rangle \rin H$ and 
$\langle z, S' \rangle \rin H'$.  To prove: $S=_{\rpw_k(A \times B)}S'$. By definition of $H$ the equivalence 
\begin{equation} \label{ekv2}
(\forall x:A)(\forall y:B)(\langle x, y \rangle \rin S \Leftrightarrow \langle \langle z,x \rangle, y \rangle \rin G)
\end{equation}
Now since $H'$ is functional, the $S''$ in (\ref{ekv1}) can only be $S'$. Putting (\ref{ekv1}) and (\ref{ekv2}) together
we have 
$$(\forall x:A)(\forall y:B)(\langle x, y \rangle \rin S \Longleftrightarrow \langle x, y \rangle \rin S')$$
as required.

\subsection*{Equalizer set} 
Let ${\bf X} = (A, X, m)$ and ${\bf Y} = (B, Y, n)$ be local sets.  Assume that
$(F,k), (G,\ell): {\bf X} \to {\bf Y}$ are two maps. Define a local set ${\bf E} = (A, E, r)$
by 
$$E =\{ a : A \;|\;  (\exists y:B) (\langle a,y \rangle \rin F \land  \langle a,y \rangle \rin G) \}.$$
here $r= \rlevel{B} \inmax k \inmax \ell$. Define the inclusion ${\bf I}=(I,p) : {\bf E} \to {\bf X}$
by 
$$E= \{ z: A \times A : \pi_1(z) \rin E \land \pi_1(z) =_A \pi_2(z) \},$$
where $p = r \inmax ||A||$.

\subsection*{Characteristic functions} The local set 
$$\Omega_k = 
(\rpw_k(\unittype), \{ x: \rpw_k(\unittype) \; | \; \top \} ,0)$$
may be considered as the collection of possible truth values of level $k$, where 
the maximal subset $${\bf t}_k = \{x: \unittype \; | \; \top \}:  \rpw_k(\unittype)$$ is the value {\em true}. Note that for $u: \rpw_k(\unittype)$, we have $u \rin \Omega_k$ and
 $u ={\bf t}_k$ if, and only if, $\theelement \rin u$. 

Let ${\bf X}=(A,X,m)$ be
an arbitrary local set, and suppose that $Y: \rpw_{m \inmax k}(A)$ satisfies $Y \subseteq X$.
We define the relation 
$$K_Y = \{ z : A \times   \rpw_k(\unittype) \; | \;   \pi_1(z) \rin Y \land \theelement \rin \pi_2(z) \}.$$
Then $K_Y : \rpw_{m \inmax k}(A  \times  \rpw_k(\unittype))$, and this gives a 
map $$\chi_Y = (K_Y, m \inmax k): {\bf X} \to \Omega_k.$$
Now for $u: A$ with $u \rin X$,
$$\langle u, {\bf t}_k \rangle \rin K_Y \Longleftrightarrow u \rin Y.$$
For ${\bf F}=(F,n): {\bf X} \to \Omega_k$ and  $u: A$ with $u \rin X$,
\begin{equation} \label{chfun}
\langle u, {\bf t}_k \rangle \rin F
\end{equation}
is a formula of level $n$. By the Functional Reducibility Axiom, $(F,n)$ is extensionally equal to some $(F', \elevel{\rpw_k(\unittype)} \inmax m \inmax 0)= (F', k \inmax m)$. Thus (\ref{chfun}) can be at most of level $m \inmax k$.  

For the special case where $m \le k$, we have that characteristic functions 
$(A,X,m) \to \Omega_k$ correspond exactly to local sets $(A,Y,k)$  with Ê$Y\subseteq X$.
In particular,  characteristic functions ${\mathbb N} \to \Omega_k$  corresponds to subsets
$Y : \rpw_k(\nattype)$.

\begin{remark} {\em
It is to be expected that IRTT gives rise to a natural notion of 
predicative topos, but we leave the detailed investigation of this
for future work. }
\end{remark}

\section{Adding classical logic to {\bf IRTT}}

Adding classical logic to {\bf IRTT} we arrive at a version of Russell's theory with
the full reducibility axiom.

\cutout{
\medskip
Assuming the Law of Excluded Middle (LEM), a power of the unit type will have just two
elements up extensionality:  {\em true} ${\bf t}_k$ and  {\em false}, given by ${\bf f}_k =_{\rm def} \{ x: {\bf 1} \; | \; \bot \}$.
\medskip
\begin{lemma} Assume (LEM). Then for any $u: \rpw_k({\bf  1})$, $u=_{\rpw_k({\bf  1})} {\bf t}_k$ or
 $u=_{\rpw_k({\bf  1})} {\bf f}_k$.
\end{lemma}
\begin{proof} Let  $u: \rpw_k({\bf  1})$. Then by (LEM) we have $\theelement \rin u$ or 
$\lnot \theelement \rin u$. In the former case $u=_{\rpw_k({\bf  1})} {\bf t}_k$. In the latter
case assume $x:{\bf 1}$ and $x \rin u$. Now $x =_{\bf 1} \theelement$, so in fact 
$\theelement \rin u$. This is a contridaction. Hence  $x \rin {\bf f}_k$. Thus $u \subseteq 
{\bf f}_k$. Trivially ${\bf f}_k \subseteq  u$, so by extensionality $u=_{\rpw_k({\bf  1})} {\bf f}_k$.
\end{proof}}

For a type $A$, let $\tilde{A}=(A, \{ x: A \; | \; \top\},0)$ be the corresponding local set.
Thus $\tilde{N} = {\mathbb N}$. A particular case of
the Functional Reducibility Axiom for $\tilde{A}$ and $\tilde{N}$  is 
that for any $r$:
\begin{align}
\label{psra}
& (\forall F: \rpw_{r}(A \times N)) \\ \notag
& \qquad \Bigl[\mbox{$F$ map from $\tilde{A}$ to $\tilde{N}$} 
\Rightarrow \\ \notag
& \qquad \qquad \qquad\quad (\exists G:  \rpw_{0}(A \times N))(\forall z:A \times N)(z \rin F \Leftrightarrow z \rin G)\Bigr]
\end{align}
The Full Reducibility Axiom (\ref{fr}) follows from the  Principle of Excluded Middle (PEM): 
\begin{thm} Assume (PEM). Then for any type $A$ and level $k$:
$$(\forall X: \rpw_k(A))(\exists Y: \rpw_0(A))(\forall x:A)(x \rin X \Leftrightarrow x \rin Y).$$
\end{thm}
\begin{proof}  Let $X: \rpw_k(A)$. Define a relation
$$F = \{ z: A \times N \; | \; \pi_1(z) \rin X \land \pi_2(z) =_N 1 \lor
                                          \lnot (\pi_1(z) \rin X)  \land \pi_2(z) =_N 0\}.$$
Then $F : \rpw_{k}(A \times N)$. Now $F$ can be seen to be a map from $\tilde{A}$ to $\tilde{N}$, where (PEM) is used to prove totality. By (\ref{psra}) 
 above we get $G : \rpw_{0}(A \times N)$ such that
$$(\forall z:A \times N)(z \rin F \Leftrightarrow z \rin G).$$
Let 
$$Y= \{x: A \; | \; \langle x, 1 \rangle \rin G\}.$$
Then $Y: \rpw_0(A)$ and clearly
$$(\forall x:A)(x \rin X \Leftrightarrow x \rin Y).$$

\end{proof}

% Suppose that ${\bf F}: {\bf X} \to \Omega_k$ is a map, where ${\bf X}=(A,X,m)$
% and ${\bf F}=(F, n)$. 

% A one element local set is $\mathbbm{1} = ({\bf 1}, \{ x: {\bf 1} \; | \; \top \},0)$.
% Define the  map ${\bf t}: \mathbbm{1} \to \Omega_k$

\end{document}